\documentclass[11pt]{article}
\usepackage{amsmath,amsthm,amsfonts}
\usepackage{fullpage}
\usepackage{dsfont}
\usepackage[colorlinks=true]{hyperref}
\usepackage{mathptmx}
\usepackage{microtype}

\newcommand{\supp}{\operatorname{supp}}
\newcommand{\C}{\mathbb{C}}
\renewcommand{\vec}[1]{\mathbf{#1}}

\DeclareMathOperator{\tr}{tr}

\DeclareMathOperator{\Res}{Res}

\DeclareMathAlphabet{\varmathbb}{U}{bbold}{m}{n}

\DeclareMathOperator*{\Exp}{\mathbb{E}}

\begin{document}

\newtheorem{theorem}{Theorem}
\newtheorem{lemma}{Lemma}
\newtheorem{claim}{Claim}
\newtheorem{proposition}{Proposition}
\newtheorem{exercise}{Exercise}
\newtheorem{corollary}{Corollary}
\theoremstyle{definition}
\newtheorem{definition}{Definition}

\newcommand{\eps}{\epsilon}
\newcommand{\Z}{{\mathbb{Z}}}
\newcommand{\F}{{\mathbb{F}}}
\newcommand{\Var}{\mathop\mathrm{Var}}
\newcommand{\Cov}{\mathop\mathrm{Cov}}
\renewcommand{\Re}{\mathop\mathrm{Re}}
\newcommand{\frob}[1]{\left\| #1 \right\|_{\rm F}}
\newcommand{\opnorm}[1]{\left\| #1 \right\|_{\rm op}}
\newcommand{\bopnorm}[1]{\bigl\| #1 \bigr\|_{\rm op}}
\newcommand{\norm}[1]{\left\| #1 \right\|}
\newcommand{\bnorm}[1]{\bigl\| #1 \bigr\|}
\newcommand{\abs}[1]{\left| #1 \right|}
\newcommand{\inner}[2]{\left\langle #1, #2 \right\rangle}

\newcommand{\wf}{\widehat{f}}
\newcommand{\wg}{\widehat{g}}
\newcommand{\wG}{\widehat{G}}
\newcommand{\wH}{\widehat{H}}
\newcommand{\wpsi}{\widehat{\psi}}
\newcommand{\wpsidag}{\widehat{\psi^\dagger}}
\newcommand{\U}{\textsf{U}}
\newcommand{\GL}{\textsf{GL}}
\newcommand{\SL}{\mathsf{SL}}
\newcommand{\fsub}{\operatorname{F}}
\newcommand{\dmin}{d_{\min}}
\newcommand{\planch}{\mathcal{P}}
\newcommand{\frack}[2]{#1 / #2}

\title{Regarding a Representation-Theoretic Conjecture of Wigderson}

\author{Cristopher Moore\\Computer Science Department\\University of New Mexico\\and the Santa Fe Institute\\{\texttt {moore@cs.unm.edu}} 
\and Alexander Russell\\Computer Science and Engineering\\University of Connecticut\\{\texttt{acr@cse.uconn.edu}}}

\maketitle

\begin{abstract}
We show that there exists a family of irreducible representations
$\rho_i$ (of finite groups $G_i$) such that, for any constant $t$, the
average of $\rho_i$ over $t$ uniformly random elements $g_1, \ldots,
g_t \in G_i$ has operator norm $1$ with probability approaching 1 as
$i \rightarrow \infty$.  This settles a conjecture of Wigderson in the negative.
\end{abstract}

\section{Introduction}

There are known families of finite groups, such as
$\SL_2(\F_p)$ as $p \rightarrow \infty$, which yield expanding Cayley
graphs over a constant number of uniformly random group elements $g_1,\ldots,g_t$~\cite{BG08}.  In particular,
for any nontrivial irreducible representation $\rho$ of such a group, the expectation of $\rho$ over the $g_i$ has operator norm bounded below $1$:
\begin{equation}
\label{eq:rep-norm}
\Exp_{\{g_1,\ldots,g_t\}} \left[ \opnorm{ \frac{1}{t} \sum_{i=1}^t \rho(g_i) } \right] \leq 1 - \epsilon\,,
\end{equation}
where $t > 1$ and $\epsilon > 0$ are constants depending only on
the group family.  

A conjecture of A. Wigderson states that there are universal constants $t$ and $\eps > 0$ such that~\eqref{eq:rep-norm} holds for any group $G$ and any nontrivial irrep $\rho \in \wG$.  This would imply that for each $\rho$, a constant number of random elements make $G$ ``act like an expander'' with respect to $\rho$, and therefore that $O(\log |\wG|)$ random elements suffice to turn $G$ into an expander.  This holds, in particular, for abelian groups: while no constant number of elements suffice to make $\Z_2^n$ into an expander, or even to generate the entire group, just $t=2$ elements suffice to bound the expected norm of any one irrep, with a universal constant $\epsilon > 0$.


Our strategy will be to work in a family of finite groups of the form $G=K^n$.  Such groups simultaneously possess high-dimensional representations and the property that any collection of a constant number of group elements generate a subgroup of constant size. We then show that a representation $\rho$ selected at random according to the Plancherel distribution contains a copy of the trivial representation when restricted to almost all relevant small subgroups.  This implies that, with high probability in the random elements $g_1,\ldots,g_t$, the expectation of $\rho$ over these elements has operator norm $1$.

Our proof focuses on the random variable 
\begin{equation}
\label{eq:xh}
X_H = \frac{\langle \Res_H \chi_\rho, 1\rangle_H}{d_\rho}\,,
\end{equation}
where $H$ is a fixed subgroup of a group $G$ and $\rho$ is selected
according to the Plancherel measure.  Since $X_H$ is the
the dimensionwise fraction of $\rho$ that restricts to the trivial
representation under $H$, whenever $X_H > 0$ we have 
\[
\opnorm{ \Exp_{h \in H} \rho(h) } = 1 \, . 
\]
We will show that for groups of the form $K^n$, if $H=\langle g_1, \ldots, g_t \rangle$ then $X_H > 0$ with high probability.  We do this by computing the first two moments of $X_H$, first for general group-subgroup pairs, and then specializing to the groups $K^n$.
 
\paragraph{Notation; actions on the group algebra}  Given a finite group $G$, let $\wG$ denote its set of irreducible representations $\rho$, let $\chi_\rho(g) = \tr \rho(g)$ denote their characters, and let $d_\rho = \chi_\rho(1)$ denote their dimensions.  Let $\C[G]$ denote the group algebra on $G$. The left action of $G$ on $\C[G]$ obtained by linearly extending the rule $g: g' \mapsto g g'$ induces the \emph{regular representation} of $G$ with character
$$
R(g) = \begin{cases} |G| & \text{if $g = 1$}\,,\\
  0 & \text{otherwise.} \end{cases}
$$
As a consequence of character orthogonality, one can express $R$ as the sum
$$
R = \sum_{\rho \in \wG} d_\rho \chi_\rho \, . 
$$
The algebra $\C[G]$ can likewise be given the structure of a $G \times G$ representation by linearly extending the rule $(g_1, g_2): g \mapsto g_1 g g_2^{-1}$. The chararacter $B$ of this representation thus satisfies
$$
B(g_1, g_2) = \abs{ \left\{ g \,:\, g_1 g g_2^{-1} = g \right\} } \,.
$$
Again applying character orthogonality, $B$ can be expressed as the sum
$$
B(g_1, g_2) = \sum_{\rho \in \wG} \chi_\rho(g_1) \cdot \chi_\rho^*(g_2) \, .
$$
Finally, the Plancherel measure $\planch$ on $\wG$ assigns each representation $\tau$ the probability mass
\[
\planch(\rho) = \frac{d_\rho^2}{|G|} \, . 
\]

Returning now to the random variable $X_H$ defined above, we compute its first two moments.

\section{The first two moments}

\paragraph{The expectation}  For the first moment, observe that if $\rho \in \wG$ is distributed according to the Plancherel measure, then 
\begin{align*}
  \Exp_\rho X_H 
  &= \Exp_\rho \frac{\langle \Res_H \chi_\rho, 1\rangle_H}{d_\rho} 
  = \sum_\rho \frac{d_\rho^2}{|G|} \frac{\langle \Res_H \chi_\rho, 1\rangle_H}{d_\rho} 
  = \frac{1}{|G|} \sum_\rho d_\rho \langle \Res_H \chi_\rho, 1\rangle_H\\
&= \frac{1}{|G|} \left\langle \Res_H \sum_\rho  d_\rho
    \chi_\rho,
    1\right\rangle_H = \frac{1}{|G|} \langle \Res_H R,
    1\rangle_H = \frac{1}{|H|} \, . 
\end{align*}

\paragraph{The second moment}
For the second moment, observe that
\begin{align*}
\Exp_\rho X_H^2 &= \sum_{\rho} \frac{d_\rho^2}{|G|} \frac{\langle \Res_H
\rho, 1\rangle_H^2}{d_\rho^2} = \frac{1}{|G|}  \sum_{\rho} \langle
\Res_{H \times H}
\chi_\rho \times \chi_\rho^* , 1\rangle_{H \times H}\\
&= \frac{1}{|G|}  \left \langle
\Res_{H \times H}
B , 1\right \rangle_{H \times
H} = \frac{1}{|G|} \frac{1}{|H|^2} \sum_{h_1, h_2 \in H} |\{ g\,:\,
h_1^{-1} g h_2 = g\}|\,.
\end{align*}
Thus
\begin{align*}
\Var_\rho[X_H] &= \Exp_\rho X_H^2 - \frac{1}{|H|^2} 
= \frac{1}{|H|^2} \left( \sum_{(h_1, h_2) \in H \times H} \Pr_{g}[h_1 = g^{-1} h_2 g] \;-\; 1 \right) \\
&= \frac{1}{|H|^2} \sum_{\substack{(h_1, h_2) \in H
    \times H\\(h_1, h_2) \neq (1,1)}} \Pr_{g}[h_1 = g^{-1} h_2 g] 
= \frac{1}{|H|^2} \sum_{\substack{h \in H\\h \neq 1}} \frac{|h^G \cap H|}{|h^G|} 
\leq \frac{1}{|H|} \sum_{\substack{h \in H\\h \neq 1}} \frac{1}{|h^G|} \,,
\end{align*}
where $h^G = \{ g^{-1} h g : g \in G\}$ is the conjugacy class of $h$ in $G$.

Applying Chebyshev's inequality, we conclude that
\begin{equation}
\label{eq:general-restriction}
\Pr_\rho[ X_H = 0 ] 
\leq \Pr\left[ \Bigl|X_H - \frac{1}{|H|}\Bigr| \geq \frac{1}{|H|}\right] 
\leq |H| \sum_{\substack{h \in H\\h \neq 1}} \frac{1}{|h^G|}\,.
\end{equation}

\section{Groups of the form $G=K^n$ and the main theorem}

We consider now the groups $G=K^n$, where $K$ is a finite group with no center.  We will prove the following.
\begin{theorem}
\label{thm:main}
Let $K$ be a finite group with center $\{1\}$ and let $G=K^n$.  Let $t$ be such that $2|K|^t \le \alpha \sqrt{n}$, let $\vec{h_1}, \ldots, \vec{h_t}$ be independent elements selected uniformly from $G$, and let $\rho = \rho_1 \otimes \cdots \otimes \rho_n$ be chosen according to the Plancherel measure on $\wG$.  Then 
  \[
  \Pr_{\rho,\{\vec{h_i}\}}[ X_H =0] 
  \leq \frac{2 \alpha}{\sqrt{n}} + \left( 2^{-1/\alpha} |K|^{\alpha}\right)^{\sqrt{n}}\,,
  \]
where $H$ is the subgroup generated by the elements $\vec{h_1},
\ldots, \vec{h_t}$ and $X_H$ is defined as in~\eqref{eq:xh}.  It
follows that there exists a representation $\rho = \rho_1 \otimes
\cdots \otimes \rho_n$ so that
\[
  \Pr_{\{\vec{h_i}\}}\left[ \Bigl\| \frac{1}{t} \sum_i \rho(\vec{h_i}) \Bigr\|_\text{op} < 1\right] 
  \leq \frac{2 \alpha}{\sqrt{n}} + \left( 2^{-1/\alpha} |K|^{\alpha}\right)^{\sqrt{n}}\,.
  \]
\end{theorem}

If $\Exp_{\vec{h} \in H} \opnorm{\rho(\vec{h})} = 1$, then certainly $\Exp_{\{ \vec{h_i}\}} \opnorm{\rho(\vec{h}_i)} = 1$.  Thus, with high probability in the group elements $\{ \vec{h}_i \}$, almost all representations $\rho$ have norm $1$ when averaged over the $\vec{h}_i$, where ``almost all'' is defined with respect to the Plancherel measure.

Note that for any fixed $K$, there is an $\alpha > 0$ such that $\Pr[X_H = 0] = o(1)$.  Thus in order to bring the expected norm of $\rho$ down to $1-\eps$ for any constant $\eps > 0$, we need $t = \Omega(\log n) = \Omega(\log \log |G|)$ random elements.  Compare this to the Alon-Roichman theorem~\cite{AR94}, which states that $O(\log |G|)$ random elements suffice with high probability to bound the norm of \emph{every} irrep $\rho$, and thus turn $G$ into an expander.  We have not attempted to close this gap.

Anticipating the proof, we collect a few facts about subgroups of $K^n$. Let $\pi_i: K^n \rightarrow K$ be the homomorphism that
projects onto the $i$th coordinate of $K^n$. For a sequence $\vec{h_1}, \ldots, \vec{h_t}$ of $t$ elements of
$K^n$, define the function $\vec{s}: \{ 1, \ldots, n\} \rightarrow K^t$ by the rule
$$
\vec{s}(i) = ( \pi_i(\vec{h_1}), \ldots, \pi_i(\vec{h_t}))\,.
$$ 
Observe that if $\vec{s}(i) = \vec{s}(j)$ then the $i$th and $j$th coordinates of any product of the $\vec{h_i}$ are identical. 
\begin{lemma}
  \label{lem:subgroup-size}
  Let $H$ be the subgroup of $K^n$ generated by $t$ elements $\vec{h_1}, \ldots, \vec{h_t}$. Then $|H| \leq |K|^{|K|^t}$.
\end{lemma}
\begin{proof}
  As $K^t$ has cardinality $|K|^t$, the coordinates of $K^n$ can be
  partitioned into no more than $|K|^t$ sets, each of which is
  constant for any product of the $\vec{h_i}$. It follows that they
  generate a group of size no more than $|K|^{|K|^t}$.
\end{proof}

For an element $\vec{g} = (g_1, \ldots, g_n) \in K^n$, let $\supp(\vec{g}) = |\{ i : g_i \neq 1\}|$ denote the size of the support of $\vec{g}$. We extend the definition to subgroups: for a subgroup $H < K^n$, define
$$
\supp(H) = \min_{\substack{\vec{h} \in H\\ \vec{h} \neq \vec{1}}} \supp(\vec{h})\,. 
$$
While the exact relationship between $\supp(H)$ and its generators may
be complicated, we can give a lower bound for $\supp(H)$ in terms of the
partition formed by the function $\mathbf{s}$ defined
above. Specifically, for a sequence of elements $\vec{h_1}, \ldots,
\vec{h_t}$, let
$$
d = d(\vec{h_1}, \ldots, \vec{h_t}) 
= \min_{\vec{s} \in K^t} \,\abs{ \left\{ i : \vec{s}(i) = \vec{s} \right\} } 
$$
denote the size of the smallest set in the partition, 
and observe that if $\vec{h}$ is a product of some sequence of the
$\vec{h_i}$ not equal to $\vec{1}$, then $\supp(\vec{h}) \geq
d$. In particular, if $H$ is the subgroup generated by the
$\vec{h_i}$, we have $\supp(H) \geq d(\vec{h_1}, \ldots, \vec{h_t})$.

\begin{lemma}
\label{lem:support}
Let $\vec{h_1}, \ldots, \vec{h_t}$ be $t$ uniformly random elements of $K^n$. Then
$$
\Pr\left[ d(\vec{h_1}, \ldots, \vec{h_t}) \leq \frac{n}{2 |K|^t}
\right] \leq \frac{4 |K|^{2t}}{n}\,.
$$
\end{lemma}
\begin{proof}
For a given element $\vec{s} \in K^t$, let $X_{\vec{s}}$ be the random variable equal to the size of the corresponding set in the partition, $\{ i : \vec{s}(i) = \vec{s} \}$.  Since $X_{\vec{s}}$ is binomially distributed as $\textrm{Bin}(n,1/|K|^t)$, we have $\Exp X_{\vec{s}} = n/{|K|^t}$ and 
$\Var X_{\vec{s}} \le n/|K|^t$.  By Chebyshev's inequality, 
  $$
  \Pr\left[ \left| X_{\vec{s}} - \frac{n}{|K|^t}\right| 
  \geq \frac{n}{2 |K|^t} \right] \leq \frac{4|K|^t}{n} \, .
  $$
Since there are $|K|^t$ elements of $K^t$, the union bound implies the statement of the lemma.
\end{proof}
\noindent
Note that the random variable $X_{\vec{s}}$ obeys the Chernoff bound, making it much more concentrated than the second moment calculation we did here suggests, but this is not important to our results.

We finally return to the proof of Theorem~\ref{thm:main}.
\begin{proof}[Proof of Theorem~\ref{thm:main}]
Let $\vec{h_1}, \ldots, \vec{h_t}$ be $t$ elements chosen independently and uniformly at random from $K^n$, and let $H$ denote the subgroup they generate. Additionally, let $\rho = \rho_1 \otimes \cdots \otimes \rho_n$ be a representation of $K^n$ selected according to the Plancherel measure.  Then with $d = d(\vec{h_1}, \ldots, \vec{h_t})$ as above,
\begin{equation}
\label{eq:pr-decomp}
\begin{split}
   \Pr_{\rho, \{\vec{h_i}\}}[ X_H =0] 
   &\leq  
   \Pr_{\{ \vec{h_i}\}}\left[ d < \frac{n}{2 |K|^t} \right] 
 + \Pr_{\rho, \{\vec{h_i}\}}\left[ X_H = 0 \;\;\vrule\;\; d \geq \frac{n}{2 |K|^t}\right] \\
    &\leq \frac{4 |K|^{2t}}{n}  + \max_{\substack{\text{$\{ \vec{h_i}\}$ such that}\\ d \geq \frack{n}{2 |K|^t}}} 
\Pr_\rho \left[ X_H = 0\right] \, ,
 \end{split}
 \end{equation}
 the second line following from Lemma~\ref{lem:support}.

As $K$ has no center, any nontrivial element $g \in K$ has at least two conjugates.  In particular, the centralizer $Z_g = \{ h \mid g = h^{-1} g h \}$ is a proper subgroup of $K$.  It follows that for an element $\vec{h} \ne 1$ of $K^n$ the conjugacy class $\vec{h}^{K^n}$ has cardinality $|\vec{h}^{K^n}| \geq 2^{\supp(\vec{h})}$.  If $d \geq \frack{n}{2 |K|^t}$, then $\supp(H) \geq \frack{n}{2 |K|^t}$ and
 $$
 |H| \sum_{\vec{h} \neq \vec{1}} \frac{1}{|\vec{h}^{K^n}|} 
 \leq |H|^2 \,2^{-\supp(H)} 
 \leq |K|^{2|K|^t} 2^{-n/2|K|^t}\,,
 $$
 by Lemma~\ref{lem:subgroup-size}.  Writing $2|K|^t \le \alpha \sqrt{n}$, we have
 $$
 |H| \sum_{\vec{h} \neq \vec{1}} \frac{1}{|\vec{h}^{K^n}|} 
 \leq \left( 2^{-1/\alpha} |K|^\alpha \right)^{\sqrt{n}} \, . 
 $$
Combining equations~\eqref{eq:general-restriction} and~\eqref{eq:pr-decomp} gives 
\[
\Pr_{\rho, \{\vec{h_i}\}}[ X_H =0] 
\leq \frac{2 \alpha}{\sqrt{n}} + \left( 2^{-1/\alpha} |K|^{\alpha}\right)^{\sqrt{n}}\,,
\]
completing the proof.
\end{proof}


\end{document}